\documentclass[a4paper,draft,reqno,12pt]{amsart}
\usepackage[utf8]{inputenc}
\usepackage[T1, T2A]{fontenc}
\usepackage[english]{babel}
\usepackage{amsmath}
\usepackage{amssymb}
\usepackage{amscd}
\usepackage{amsthm}
\usepackage{euscript}
\usepackage[all]{xy}
\newtheorem{proposition}{Proposition}

\newtheorem{lemma}{Lemma}
\newtheorem{theorem}{Theorem}

\newtheorem{corollary}{Corollary}
\theoremstyle{definition}

\newtheorem{example}{Example}
\theoremstyle{remark}
\newtheorem {remark}{Remark}

\DeclareMathOperator{\Spec}{Spec}

\DeclareMathOperator{\Aut}{Aut}

\DeclareMathOperator{\Cl}{Cl}

\DeclareMathOperator{\codim}{codim}

\DeclareMathOperator{\ppr}{pr}%

\DeclareMathOperator{\Ker}{Ker}

\def\GG{{\mathbb G}}
\def\FF{{\mathbb F}}

\def\KK{{\mathbb K}}
\def\TT{{\mathbb T}}
\def\ZZ{{\mathbb Z}}

\def\QQ{{\mathbb Q}}
\def\PP{{\mathbb P}}
\def\AA{{\mathbb A}}

\def\mm{\mathfrak{m}}

\renewcommand{\phi}{\varphi}
\renewcommand{\ge}{\geqslant}
\renewcommand{\le}{\leqslant}
\newcommand{\Zgtzero}{\mathbb{Z}_{>0}}
\newcommand{\Zgezero}{\mathbb{Z}_{\geqslant 0}}
\newcommand{\pa}{\partial}
\newcommand{\ti}{\tilde}

\begin{document}
\date{}
\title[Commutative algebraic monoids on $\AA^n$]{Commutative algebraic monoid structures \\ on affine spaces}
\author{Ivan Arzhantsev}
\address{National Research University Higher School of Economics, Faculty of Computer Science, Kochnovskiy Proezd 3, Moscow, 125319 Russia}
\email{arjantsev@hse.ru}

\author{Sergey Bragin}
\address{National Research University Higher School of Economics, Faculty of Computer Science, Kochnovskiy Proezd 3, Moscow, 125319 Russia}
\email{sbragin@hse.ru}

\author{Yulia Zaitseva}
\address{National Research University Higher School of Economics, Faculty of Computer Science, Kochnovskiy Proezd 3, Moscow, 125319 Russia}
\email{yuliazaitseva@gmail.com}

\subjclass[2010]{Primary 20M14, 20M32; \ Secondary 14R20, 20G15}

\keywords{Algebraic monoid, group embedding, grading, locally nilpotent derivation, local algebra, toric variety, Cox ring, Demazure root, additive action}

\thanks{The first and the third authors were supported by RSF grant 19-11-00172.}

\begin{abstract}
We study commutative associative polynomial operations $\AA^n\times\AA^n\to\AA^n$ with unit
on the affine space $\AA^n$ over an algebraically closed field of characteristic zero.
A~classification of such operations is obtained up to dimension 3. Several series
of operations are constructed in arbitrary dimension. Also we explore a connection between
commutative algebraic monoids on affine spaces and additive actions on toric varieties.
\end{abstract}

\maketitle


\section*{Introduction}
\label{sec0}

An (affine) algebraic monoid is an irreducible (affine) algebraic variety $S$ with an associative multiplication
$$
\mu \colon S\times S\to S,\quad (a,b)\mapsto ab,
$$
which is a morphism of algebraic varieties, and a unit element $e\in S$ such that ${ea=ae=a}$ for all $a\in S$. Examples of affine algebraic monoids are affine algebraic groups and multiplicative monoids of finite dimensional associative algebras with unit. The group of invertible elements $G(S)$ of an algebraic monoid $S$ is open in $S$. Moreover, $G(S)$ is an algebraic group.
By~\cite[Theorem~3]{Ri2}, every algebraic monoid $S$, whose group of invertible elements $G(S)$ is
an affine algebraic group, is an affine monoid. An affine algebraic monoid $S$ is called reductive if the group $G(S)$ is a reductive affine algebraic group.

By a group embedding we mean an irreducible affine variety $X$ with an open embedding $G\hookrightarrow X$ of an affine algebraic group $G$ such that both actions by left and right multiplications of $G$ on itself can be extended to $G$-actions on $X$. In other words, the variety $X$ is a $(G\times G)$-equivariant open embedding of the homogeneous space $(G\times G)/\Delta(G)$, where $\Delta(G)$ is the diagonal in $G\times G$.

Any affine monoid $S$ defines a group embedding $G(S)\hookrightarrow S$. The converse statement claims that for every group embedding $G\hookrightarrow S$ there exists a structure of an affine algebraic monoid on $S$ such that the group $G$ coincides with the group of invertible elements $G(S)$. This is proved in~\cite[Theorem~1]{Vi} under the assumption that $G$ is reductive and in~\cite[Proposition~1]{Ri1} for arbitrary~$G$.

Nowadays the theory of affine algebraic monoids and group embeddings is a rich and deeply developed area of mathematics lying at the intersection of algebra, algebraic geometry, combinatorics and representation theory; we refer to~\cite{Pu5,Re,Ri1,Vi} for general presentations.

\smallskip

In this paper we study commutative algebraic monoid structures on affine spaces. The ground field $\KK$ is an algebraically closed field of characteristic zero. The theory of commutative
reductive algebraic monoids, i.e. algebraic monoids with an algebraic torus as the group of invertible elements, is nothing but the theory of affine toric varieties; see~\cite{Neeb} for
more information on toric monoids. We concentrate on non-reductive commutative monoids.

There were several research projects in this direction. For example, all algebraic semigroup structures on the affine line are classified in~\cite[Theorem~1]{Yo}. In~\cite{Pu1}, algebraic semigroups on affine spaces are studied under the assumption that the multiplication is given by polynomials of degree at most 2.

In Section~\ref{sec2} we collect basic facts on commutative algebraic monoids. Let us define the rank of a commutative monoid $S$ as the dimension of the maximal torus of the group $G(S)$.
Section~\ref{sec3} contains a classification of commutative monoids on $\AA^n$ of rank $0$,
$n-1$ and $n$. In particular, this gives a classification of commutative monoid structures on $\AA^n$ for $n\le 2$ (Proposition~\ref{A1A2}).

In Section~\ref{sec4} we classify commutative monoid structures on $\AA^3$ (Theorem~\ref{A3}). After some reductions we come to the classification of pairs of commuting homogeneous locally nilpotent derivations of degree zero on a positively graded polynomial algebra $\KK[x_1,x_2,x_3]$.
This classification is obtained in Proposition~\ref{prop2der}.

Section~\ref{sec5} is devoted to a particular class of algebraic monoids on affine spaces. Namely, we consider multiplicative monoids of finite dimensional associative algebras with unit and call such monoids bilinear. Proposition~\ref{propmul} claims that an algebraic monoid $S$ on $\AA^n$ is bilinear if and only if the action of the group $G(S)\times G(S)$ on $\AA^n$ is linearizable. While for $n \le 3$ structures of a commutative monoid on $\AA^n$ are parameterized by discrete parameters (Theorem~\ref{A3}), starting from $n=7$ there are continuous families of different commutative bilinear monoid structures on $\AA^n$ (Proposition~\ref{pinfp}). Such monoids correspond to local algebras and thus have rank 1.

In Section~\ref{sec6} we come to a geometric construction related to commutative monoids. Let us recall that an additive action on a normal variety $X$ is a regular faithful action $\GG_a^s\times X\to X$ of a commutative unipotent group $\GG_a^s$ with an open orbit. We show that every additive action on a toric variety gives rise to a commutative monoid structure on an affine space. In~order to obtain the desired group embedding into affine space, one should lift the additive action to the spectrum of the Cox ring of the toric variety and to extend this action by the characteristic torus. The~corresponding commutative monoid is bilinear if and only if the toric variety is a product of projective spaces (Proposition~\ref{assact}). Commutative monoid structures coming from additive actions on Hirzebruch surfaces provide interesting examples of commutative monoids on $\AA^4$ (Example~\ref{exref}). Conversely, classification results from Section~\ref{sec4} allow to show that the weighted projective plane $\PP(1,b,c)$ admits precisely two additive actions (Proposition~\ref{WPP}). It is an interesting problem to classify additive actions on weighted projective spaces of arbitrary dimension. We do not know whether there is a "weighted version" of the Hassett-Tschinkel correspondence between additive actions on $\PP^n$ and local algebras of dimension $n+1$, see~\cite{HT}.

In the last section we briefly describe the well-known structure theory of affine monoids in the particular case of commutative monoids on affine spaces. This case is much simpler, and we give short proofs whenever it is possible. We show that every commutative monoid of rank $r$ on an affine space contains at least $2^r$ idempotents, and the number of idempotents is precisely $2^r$ provided $r\le 2$ (Proposition~\ref{us}). For commutative monoids of rank 1 with zero this implies that every element is either invertible or nilpotent.

The autors are grateful to the referee for a careful reading and to Sergey Dzhunusov who pointed our attention to an inaccuracy in Example~\ref{exref}. 


\section{Commutative algebraic monoids}
\label{sec2}

Let us begin with some basic definitions. We say that a monoid $S$ is a monoid with zero if there is an element $0\in S$ such that $0a=a0=0$ for any $a\in S$.
An ideal of a monoid $S$ is a nonempty subset $I\subseteq S$ such that $SIS\subseteq I$.
If there exists an ideal $K\subseteq S$ contained in all ideals of $S$, we say that $K$ is the
kernel of $S$. Every algebraic monoid $S$ contains a unique closed $G(S)\times G(S)$-orbit and this orbit is the kernel of $S$~\cite[Theorem~1]{Ri1}. In particular, $S$ is a monoid with zero if and only if its kernel consists of one element.

An affine monoid is called commutative if the multiplication $S \times S \to S$ is commutative, or, equivalently, the group of invertible elements $G(S)$ is a commutative algebraic group. Denote by $\GG_m$ and $\GG_a$ the multiplicative and the additive groups of the ground field $\KK$, respectively. It is well known that any connected commutative affine algebraic group is isomorphic to $\GG_m^r \times \GG_a^s$ for some non-negative integers $r$ and $s$; see \cite[Theorem 15.5]{Hum}. We say that $r$ is the rank of the group~$G$ and $s$ is the corank of $G$. The rank and the corank of a commutative monoid $S$ are defined as the rank and the corank
of the group $G(S)$.

The case $s = 0$ corresponds to $G(S)=\GG_m^n$, and affine monoids in this case are precisely
affine toric varieties; see~\cite{Neeb}. For more details on affine toric varieties, see~\cite[Chapter~1]{CLS}, \cite[Section~1.3]{Fu}.

\begin{lemma} \label{l1l}
Let $S$ be a commutative monoid on $\AA^n$ and $T$ be the maximal torus in $G(S)$. The following conditions are equivalent.
\begin{enumerate}
\item
The monoid $S$ has zero.
\item
The group $G(S)$ acts on $S$ with a unique fixed point.
\item
The categorical quotient $S/\!/T:=\Spec\KK[S]^T$ is a point.
\end{enumerate}
\end{lemma}

\begin{proof}
Conditions~(1) and~(2) are equivalent for any algebraic monoid $S$. Indeed, the group $G(S)$ acts on $S$ with a unique closed orbit, and $S$ is a monoid with zero if and only if this orbit is a point.

Condition~(3) means that the algebra of invariants $\KK[S]^T$ consists of constant functions or, equivalently, there is a unique closed $T$-orbit $\mathcal{O}$ in $S$. Clearly, the orbit $\mathcal{O}$ is $G(S)$-invariant. Moreover, $\mathcal{O}$ is a factor group of $T$, thus it is a torus. By~\cite[Theorem~6.7]{PV}, the variety $S$ is a homogeneous fiber bundle over $\mathcal{O}$. Since the algebra $\KK[\mathcal{O}]$ is generated by invertible functions and
all invertible functions in $\KK[S] = \KK[x_1, \ldots, x_n]$ are constant, we conclude that $\mathcal{O}$ is a point. This implies condition~(2).

Conversely, assume that $G(S)$ acts on $S$ with a unique fixed point $p$, and the quotient $S/\!/T$ has positive dimension. Consider the quotient morphism $\pi\colon S\to S/\!/T$ induced by the inclusion $\KK[S]^T\subseteq\KK[S]$. The unipotent part $\GG_a^s$ of the group $G(S)$ commutes with the torus $T$, and its action on $S$ descends to an action on $S/\!/T$ with an open orbit.
Since all orbits of an action of a unipotent group on an affine variety are closed~\cite[Section~1.3]{PV}, the action of $\GG_a^s$ on $S/\!/T$ is transitive.
We conclude that the group $\GG_a^s$ can move the point $p$ to any fiber of the map $\pi$, a contradiction.
\end{proof}

The next lemma is an analogue of a decomposition for reductive monoids \cite[Corollary~of~Theorem~2]{Vi}.

\begin{lemma}\label{l2l}
Every commutative monoid $S$ on $\AA^n$ is isomorphic to a direct product of monoids  $\GG_a^k\times S_0$, where $S_0$ is a commutative monoid with zero on $\AA^{n-k}$.
\end{lemma}

\begin{proof}
Consider the quotient morphism $\pi\colon S\to S/\!/T$ and the induced transitive action of
the unipotent part $\GG_a^s$ of $G(S)$ on $S/\!/T$. Let $G_0$ be the stabilizer in $G(S)$ of a fiber of the morphism~$\pi$. The subgroup $G_0$ contains $T$ and the group $G(S)$ is a direct product $G_0\times\GG_a^k$ with some complementary subgroup $\GG_a^k$. The action of the latter subgroup on $S/\!/T$ is transitive and effective. Consider the fiber $S_1$ of the morphism $\pi$ containing the unit of~$S$. The morphism $\psi\colon\GG_a^k\times S_1\to S$, $(g,y)\mapsto gy$ is bijective. Since the variety $S$ is normal, $\psi$ is an isomorphism. This shows that $S_1$ is a smooth irreducible affine variety. The torus $T$ acts on $S_1$ with a unique closed orbit. Since all invertible functions on $S$ and on $S_1$ are constant, this orbit is a fixed point. By~\cite[Corollary~6.7]{PV}, the variety $S_1$ is an affine space~$\AA^{n-k}$. Moreover, since the group $G_0$ acts on $S_1$ with an open orbit, the variety $S_1$ has a structure of a commutative monoid such that the map $\psi$ is an isomorphism of monoids. By Lemma~\ref{l1l}, $S_1$ is a monoid with zero.
\end{proof}

Recall that an action $G\times\AA^n\to\AA^n$, $(g,x)\mapsto gx$ of an affine algebraic group $G$ is linearizable if there is an automorphism $\varphi\colon\AA^n\to\AA^n$ such that the conjugated action $(g,x)\mapsto\varphi(g\varphi^{-1}(x))$ is linear.

\begin{lemma}\label{l3l}
Let $S$ be a commutative monoid on $\AA^n$ and $T$ be the maximal torus in $G(S)$. Then the action of $T$ on $S$ is linearizable.
\end{lemma}

\begin{proof}
Lemma~\ref{l2l} reduces the general case to the case when $S$ contains a unique closed $T$-orbit which is a fixed point. Such actions are known to be linearizable~\cite[Corollary~6.7]{PV}.
\end{proof}

\smallskip

Now we introduce some algebraic technique that will be used in the classification of commutative algebraic monoids. Let $R$ be an algebra over~$\KK$. A $\KK$-linear map $\delta\colon R \to R$ is called a derivation of the algebra $R$ if it satisfies the Leibniz rule, that is $\delta(fg) = \delta(f)g + f \delta(g)$ for any $f, g \in R$. A derivation $\delta$ is said to be locally nilpotent if for any $f \in R$ there exists $k \in \Zgezero$ such that $\delta^k(f) = 0$.

Let $K$ be an abelian group. An algebra $R$ is said to be $K$-graded if
$$
R = \bigoplus \limits_{w \in K} R_w
$$
and $R_{w_1}R_{w_2} \subseteq R_{w_1+w_2}$ for any $w_1, w_2 \in K$. Elements of $R_w$ are called homogeneous of degree~$w \in K$. A derivation $\delta$ is called homogeneous if it maps homogeneous elements to homogeneous ones. If in addition $R$ is a domain, there exists an element $\deg \delta \in K$ such that $\delta(R_{w}) \subseteq R_{w + \deg \delta}$ for any $w \in K$. The element $\deg \delta$ is called the degree of the derivation~$\delta$.

Let $X$ be an irreducible affine variety. It is well known that $\GG_a$-actions on $X$ are in bijection with locally nilpotent derivations on the algebra $\KK[X]$; see~\cite[Section 1.5]{Fr}. On the other hand, $\GG_m^r$-actions on $X$ correspond to $\ZZ^r$-gradings on $\KK[X]$.
Commuting $\GG_a$-actions on $X$ correspond to commuting locally nilpotent derivations on $\KK[X]$. A $\GG_a$-action commutes with a $\GG_m^r$-action if and only if the corresponding locally nilpotent derivation is homogeneous of degree zero with respect to the corresponding $\ZZ^r$-grading. Thus we come to  the following statement.

\begin{lemma} \label{grLND}
A structure of a commutative monoid of rank $r$ and corank $s$ on an irreducible affine variety $X$ is defined by a $\ZZ^r$-grading on $\KK[X]$ and a collection $\delta_1,\ldots,\delta_s$ of pairwise commuting homogeneous locally nilpotent derivations on $\KK[X]$ of degree zero such that the corresponding action of the group $\GG_m^r \times \GG_a^s$ on $X$ is effective and has an open orbit.
\end{lemma}


\section{Three particular cases}
\label{sec3}
We keep notation of Section~\ref{sec2}. In this section we describe commutative monoids on affine space $\AA^n$ of arbitrary dimension $n$ in the cases $r = 0$, $s = 0$, and $s = 1$.

\begin{proposition} \label{corank01n}
Let $G = \GG_m^r \times \GG_a^s$ and $n=r+s$. A commutative monoid on $\AA^n$ with $G$ as a group of invertible elements is isomorphic to
$$
\begin{aligned}
1) \; (x_1, \ldots, x_n) * (y_1, \ldots, y_n) &= (x_1 + y_1, \ldots, x_n + y_n) \quad\text{if }r = 0;
\\
2) \; (x_1, \ldots, x_n) * (y_1, \ldots, y_n) &= (x_1 y_1, \ldots, x_n y_n) \quad\text{if }s = 0;
\\
3) \; (x_1, \ldots, x_n) * (y_1, \ldots, y_n) &= (x_1 y_1, \; \ldots, \; x_{n-1} y_{n-1}, \; x_1^{b_1} \ldots x_{n-1}^{b_{n-1}} y_n + y_1^{b_1} \ldots y_{n-1}^{b_{n-1}} x_n),
\\
\text{for some } b_i \in \Zgezero, \; b_1 &\le \ldots \le b_{n-1}, \quad\text{if }s = 1.
\end{aligned}
$$
Moreover, the corresponding monoids with different values of parameters are non-isomorphic.
\end{proposition}

\begin{proof}
1) The structure of a commutative monoid of rank 0 defines an action of the group $\GG_a^n$ on $\AA^n$ with an open orbit. Any orbit of a unipotent group on an affine variety is closed, see \cite[Section 1.3]{PV}. Thus the monoid $\AA^n$ coincides with the group $\GG_a^n$.

2) The case $s = 0$ follows from the uniqueness of the structure of a toric variety on $\AA^n$, see~\cite{BB}.

3) Let $X$ be a toric variety with the acting torus $\TT$. Consider an action $\GG_a \times X \to X$ normalized by $\TT$, that is the $\GG_a$-subgroup is stable under conjugation by elements of the torus $\TT$ in the automorphism group $\Aut(X)$. Then $\TT$ acts on $\GG_a$ by conjugation with some character $\chi$, which is called a Demazure root of $X$. If $T = \Ker \chi$ then the group $G = T \times \GG_a$ has corank~1, acts on $X$ with an open orbit, and $X$ is a $G$-embedding, see  \cite[Proposition~6]{AK}. By \cite[Theorem~2]{AK}, all $G$-embeddings of connected commutative affine algebraic group $G$ of corank~1 can be realized this way.

Let us apply this fact to $X = \AA^n$. Since the structure of a toric variety is unique up to isomorphism, we can assume that the torus $\TT$ acts on $\AA^n$ by diagonal matrices. Let $t = (t_1, \ldots, t_n) \in \TT$. Any Demazure root of $\AA^n$ has the form $\chi(t) = t_1^{b_1} \ldots t_{n-1}^{b_{n-1}}t_n^{-1}$ up to order of variables for some $b_i \in \Zgezero$, $b_1 \le \ldots \le b_{n - 1}$; see e.g.~\cite[Example~1]{AK}.

For $c = (c_1, \ldots, c_n)$ denote $x^c := x_1^{c_1} \ldots x_n^{c_n}$.  Let $b = (b_1, \ldots, b_{n-1}, 0)$. Then an element $t \in \Ker \chi$ acts on $x \in \AA^n$ by
$(t_1 x_1, \; \ldots, \; t_{n-1} x_{n-1}, \; t^b x_n)$. The $\GG_a$-action corresponding to $\chi$ has the form $\exp \alpha \delta$, $\alpha \in \GG_a$, with $\delta = x^b \frac{\pa}{\pa x_n}$, that is
$$
\exp \alpha \delta\circ x = \bigl(x_1, \ldots, x_{n-1}, \;
x_n + \alpha x^b\bigr).
$$
Thus the action of the group $G = T \times \GG_a$ is given by
\begin{equation} \label{eqcorank1act}
t \exp \alpha \delta\circ x
= \left(t_1 x_1, \ldots, t_{n-1} x_{n-1}, \; t^b \bigl(x_n + \alpha x^b\bigr)\right).
\end{equation}
This action has an open orbit $\{x_i \ne 0 \mid 1 \le i \le n-1\}$. Then a $G$-embedding can be given by the identification
$$\begin{aligned}
G = \GG_m^{n-1} \times \GG_a & \longleftrightarrow \{x_i \ne 0 \mid 1 \le i \le n-1\} \subset \AA^n;
\\
(t_1, \ldots, t_{n-1}, \alpha) & \longleftrightarrow y = (t_1, \ldots, t_{n-1}, \alpha) \circ (1, \ldots, 1, 0) =
\bigl(t_1, \; \ldots, \; t_{n-1}, \; t^b \alpha\bigr).
\end{aligned}$$
Substituting $t_i = y_i$, $1 \le i \le n - 1$, and $\alpha = y^{-b}y_n$ in~\eqref{eqcorank1act} we obtain
$$\begin{aligned}
x * y = \bigl(y_1 x_1, \; \ldots, \; y_{n-1} x_{n-1}, \;
y^b \bigl(x_n + y^{-b} y_n x^b\bigr)\bigr) =
\left(x_1 y_1, \; \ldots, \; x_{n-1} y_{n-1}, \; x^b y_n + y^b x_n\right).
\end{aligned}$$

Let us prove that these monoids are non-equivalent for different vectors $b$. Note that the complement of the open orbit consists of $n - 1$ components $X_i := \{x_i = 0\}$, $1 \le i \le n - 1$. The restriction of action~\eqref{eqcorank1act} on $X_i$ is given by
$$
t \exp \alpha \delta\circ (x_1, \ldots, 0, \ldots, x_n) =
\left\{\begin{aligned}
&\left(t_1 x_1, \ldots, 0, \ldots, t_{n-1} x_{n-1}, \; t^b \bigl(x_n + \alpha x^b\bigr) \right)
\quad\text{if } b_i = 0
\\
&\left(t_1 x_1, \ldots, 0, \ldots, t_{n-1} x_{n-1}, \; t^b x_n \right)
\quad\text{if } b_i \ne 0
\end{aligned}\right.$$
The kernel of non-effectivity of this action is a one-dimensional torus if $b_i = 0$ and the direct product of $\KK$ and the cyclic group of order $b_i$ otherwise. Thus the integers $b_i$ are uniquely determined by the monoid structure.
\end{proof}

Since all monoid structures on $\AA^1$ and $\AA^2$ are covered by Proposition~\ref{corank01n}, we obtain the following classification result.

\begin{proposition} \label{A1A2}
1) Every commutative monoid on $\AA^1$ is isomorphic to one of the following monoids:
$$\begin{array}{rl}
A: &
(x_1) * (y_1) = (x_1 + y_1);
\\
M: &
(x_1) * (y_1) = (x_1 y_1).
\end{array}$$
2) Every commutative monoid on $\AA^2$ is isomorphic to one of the following monoids:
$$\begin{array}{rl}
2A: &
(x_1, x_2) * (y_1, y_2) = (x_1 + y_1, \; x_2 + y_2);
\\
M \underset{b}{+} A: &
(x_1, x_2) * (y_1, y_2) = (x_1 y_1, \; x_1^b y_2 + y_1^b x_2), \quad b \in \Zgezero;
\\
2M: &
(x_1, x_2) * (y_1, y_2) = (x_1 y_1, \; x_2 y_2).
\end{array}$$
\end{proposition}

\begin{remark}
The operation $M \underset{b}{+} A$ decomposes into the sum of $M$ and $A$ if $b = 0$.
\end{remark}


\section{Commutative monoid structures on the affine 3-space}
\label{sec4}

For $b, c \in \Zgtzero$, $b\le c$, denote by $Q_{b, c}$ the polynomial
$$
Q_{b, c}(x_1, y_1, x_2, y_2) =
\sum \limits_{k = 1}^d \binom{d+1}{k} x_1^{e + b(k - 1)} y_1^{e+b(d-k)} x_2^{d-k+1} y_2^k,
$$
where $c = b d + e, \; d, e \in \ZZ, \; 0 \le e < b$. Note that $Q_{b, c}(x_1, y_1, x_2, y_2) = Q_{b, c}(y_1, x_1, y_2, x_2)$ and
$$
Q_{b, c}(x_1, y_1, x_2, y_2) =
\frac{(x_1^by_2+y_1^bx_2)^{d+1}-(x_1^by_2)^{d+1}-(y_1^bx_2)^{d+1}}
{x_1^{b-e} y_1^{b-e}}.
$$

Let us formulate the main result.

\renewcommand\arraystretch{1.2}

\begin{theorem} \label{A3}
Every commutative monoid on $\AA^3$ is isomorphic to one of the following monoids:
$$
\begin{array}{c|c|c}
\text{rk} & \text{Notation} & (x_1, x_2, x_3) * (y_1, y_2, y_3) \\\hline\hline
0 & 3A &
(x_1 + y_1, \, x_2 + y_2, \, x_3 + y_3)
\\
1 & M \underset{b}{+} A \underset{c}{+} A &
\bigl(x_1y_1, \, x_1^b y_2 + y_1^b x_2, \, x_1^c y_3 + y_1^c x_3\bigr),
\; b, c \in \Zgezero, \; b \le c
\\
1 & M \underset{b}{+} A \underset{b,c}{+} A &
\bigl(x_1y_1, \, x_1^b y_2 + y_1^b x_2, \, x_1^c y_3 + y_1^c x_3 + Q_{b, c}(x_1, y_1, x_2, y_2)\bigr), b, c \in \Zgtzero, b \le c
\\
2 & M + M \underset{b, c}{+} A &
(x_1 y_1, \, x_2 y_2, \, x_1^bx_2^c y_3 + y_1^b y_2^c x_3), \; b, c \in \Zgezero, \; b \le c
\\
3 & 3M &
(x_1 y_1, \, x_2 y_2, \, x_3 y_3)
\end{array}
$$
Moreover, every two monoids of different types or of the same type with different values of parameters from this list are non-isomorphic.
\end{theorem}

\medskip

\begin{example} The monoid $M \underset{b}{+} A \underset{b,c}{+} A$ with $b =c =1$ is given by
$$
x*y=\bigl(x_1y_1, \; x_1y_2 + y_1x_2, \; x_1y_3 + 2x_2y_2 + y_1x_3\bigr),
$$
while the monoid $M \underset{b}{+} A \underset{b,c}{+} A$ with $b = 1$, $c = 3$ corresponds to multiplication
$$
x*y=\bigl(x_1y_1, \; x_1 y_2 + y_1 x_2, \; x_1^3 y_3 + 4 x_1^2x_2y_2^3 + 6 x_1y_1x_2^2y_2^2 + 4y_1^2x_2^3y_2 + y_1^3 x_3\bigr).
$$
\end{example}

\medskip

We come to the proof of Theorem~\ref{A3}. First we need the following facts on locally nilpotent derivations.

\begin{lemma} \label{lem0}
Let $\delta$ be a locally nilpotent derivation on the algebra $\KK[x_1, \ldots, x_n]$ with $\delta(\KK[x_1, \ldots, x_{n-1}]) \subseteq \KK[x_1, \ldots, x_{n-1}]$ and $\delta(x_n) \in \KK x_n + \KK[x_1, \ldots, x_{n-1}]$. Then $\delta(x_n) \in \KK[x_1, \ldots, x_{n-1}]$.
\end{lemma}

\begin{proof}
Let $\delta(x_n) \in \lambda x_n + \KK[x_1, \ldots, x_{n-1}]$. Then $\delta^k(x_n) \in \lambda^k x_n + \KK[x_1, \ldots, x_{n-1}]$ is equal to~$0$ for some $k$, whence $\lambda = 0$.
\end{proof}

The following lemma is proved, for example, in~\cite[Corollary 1.20]{Fr}.

\begin{lemma} \label{lemFr}
Let $\delta$ be a locally nilpotent derivation on a domain $R$ and $f \in R$. Then $f \mid \delta(f)$ implies $\delta(f) = 0$.
\end{lemma}

Consider a graded algebra $\KK[x_1, x_2, x_3]$ with homogeneous $x_1, x_2, x_3$ of coprime degrees $a, b, c \in \Zgtzero$ respectively, $a \le b \le c$.

\begin{proposition}
\label{prop2der}
Let $\delta_1, \delta_2$ be two commuting locally nilpotent derivations of degree zero on $\KK[x_1, x_2, x_3]$ with
$\dim\bigl(\Ker\delta_1 \cap \Ker\delta_2 \cap \langle x_1, x_2, x_3\rangle\bigr) \le 1$.
Then $a = 1$ and after a homogeneous change of variables in $\KK[x_1, x_2, x_3]$ the derivations $\delta_1, \delta_2$ have one of the following forms:
$$\begin{aligned}
\text{Type 1.}\quad &\delta_i\colon\;
x_1 \;\mapsto\; 0, \quad
x_2 \;\mapsto\; \beta_i x_1^b, \quad
x_3 \;\mapsto\; \gamma_i x_1^c + \beta_i x_1^e x_2^d,
\quad\quad i = 1, 2;
\\
\text{Type 2.}\quad &\delta_i\colon\;
x_1 \;\mapsto\; 0, \quad
x_2 \;\mapsto\; \beta_i x_1^b, \quad
x_3 \;\mapsto\; \gamma_i x_1^c,
\quad\quad i = 1, 2,
\end{aligned}$$
where $c = bd + e$, $0 \le e < b$, $\beta_i, \gamma_i \in \KK$.
\end{proposition}

\begin{proof}
{\bf Step 1.} Let us prove that after a linear change of variables in $\KK[x_1, x_2, x_3]$ the derivations $\delta_1, \delta_2$ have the form
\begin{equation} \label{eqchange1}
\delta_i\colon x_1 \mapsto 0, \quad
x_2 \mapsto \beta_i x_1^b, \quad
x_3 \mapsto P_{b, c}^{(i)}(x_1, x_2),
\quad\quad i = 1, 2,
\end{equation}
where $P_{b, c}^{(i)}$ is a polynomial of the form $\sum \limits_{l = 0}^d \xi_{i l} x_1^{c - bl} x_2^l$. We proceed with four cases.

\emph{Case A.} $a = b = c$

Since $a, b, c$ are coprime we have $a = b = c = 1$. The condition $\deg \delta_i = 0$, $i = 1, 2$, implies $\delta_i(x_1), \, \delta_i(x_2), \, \delta_i(x_3) \in \langle x_1, x_2, x_3\rangle$. This means that each derivation $\delta_i$ induces a linear transformation of $\langle x_1, x_2, x_3\rangle$, which is nilpotent as $\delta_i$ is locally nilpotent. Since the two induced linear transformations are nilpotent and commute, there is a change of basis $x_1, x_2, x_3$ that makes the corresponding transformation matrices strongly upper triangular simultaneously. Thus we can assume that
$$
\begin{array}{rl}
\delta_i\colon
& x_1 \mapsto 0 \\
& x_2 \mapsto \beta_i x_1 \\
& x_3 \mapsto \gamma_i x_1 + \xi_i x_2
\end{array},
\quad\quad \beta_i,\gamma_i,\xi_i\in \KK, \quad i = 1,2.
$$

\emph{Case B.} $a = b < c$

Note that $\delta_i(x_1), \, \delta_i(x_2) \in \langle x_1, x_2\rangle $. As in Case A, the derivations $\delta_i$ induce commuting nilpotent linear transformations of $\langle x_1, x_2\rangle$ and after a change of basis $x_1, x_2$ we have
$$\delta_i\colon
x_1 \mapsto 0, \quad
x_2 \mapsto \beta_i x_1,
\quad\quad i = 1, 2.$$

If $a \nmid c$ then $\delta_i(x_3) \in \KK[x_1, x_2, x_3]_c = \langle x_3\rangle$, where $\KK[x_1, x_2, x_3]_c$ denotes the homogeneous component of degree~$c$. Lemma~\ref{lemFr} implies $\delta_i(x_3) = 0$. Then $x_1, x_3 \in \Ker\delta_1 \cap \Ker\delta_2$, which contradicts the assumptions of the proposition.

If $a \mid c$ then $a = b = 1$ and $\delta_i(x_3) \in \KK[x_1, x_2, x_3]_c = \langle x_3\rangle \oplus \langle x_1^{c-l} x_2^l \mid 0 \le l \le c\rangle$. By Lemma~\ref{lem0} we have
$$
\delta_i\colon
x_1 \mapsto 0, \quad
x_2 \mapsto \beta_i x_1, \quad
x_3 \mapsto \sum\limits_{l=0}^c \xi_{i l} \, x_1^{c-l} x_2^l,
\quad\quad i = 1, 2.$$

\smallskip

\emph{Case C.} $a < b = c$

From $a < b$ we have $\delta_i(x_1) \in \langle x_1\rangle$, which implies $\delta_i(x_1) = 0$ by Lemma~\ref{lemFr}.

Let $a \nmid b$. Then $\delta_i(x_2), \, \delta_i(x_3) \in \langle x_2, x_3\rangle$ and the derivations $\delta_i$ induce linear transformations of $\langle x_2, x_3\rangle$. As above, there is a change of basis in $\langle x_2, x_3\rangle$ such that
$$
\delta_i\colon
x_1 \mapsto 0, \quad
x_2 \mapsto 0, \quad
x_3 \mapsto \xi_i x_2,
\quad\quad i = 1, 2,$$
which implies $x_1, x_2 \in \Ker \delta_1 \cap \Ker \delta_2$, a contradiction.

If $a \mid b$ then $a = 1$. Note that $\delta_i(x_2), \, \delta_i(x_3) \in \langle x_1^b, x_2, x_3\rangle$, and $\delta_i(x_1^b) = 0$. It follows that the derivations $\delta_i$ induce commuting nilpotent linear transformations of $\langle x_1^b, x_2, x_3\rangle$. Their projections $\ppr \delta_i$ on $\langle x_2, x_3\rangle$ commute and are nilpotent as well. After a change of basis in $\langle x_2, x_3\rangle$ we obtain
$$
\ppr \delta_i\colon
x_2 \mapsto 0, \quad
x_3 \mapsto \xi_i x_2,
\quad\quad i = 1, 2.$$
For initial derivations this means
$$
\delta_i\colon
x_1 \mapsto 0, \quad
x_2 \mapsto \beta_i x_1^b, \quad
x_3 \mapsto \gamma_i x_1^b + \xi_i x_2,
\quad\quad i = 1, 2.$$

\smallskip

\emph{Case D.} $a < b < c$

As in Case C, the condition $a < b$ implies $\delta_i(x_1) = 0$, $i = 1, 2$. If $a \nmid b$ then $\delta_i(x_2) \in \langle x_2\rangle$, whence by Lemma~\ref{lemFr} we see that $\delta_i(x_2) = 0$. In this case, $x_1, x_2 \in \Ker\delta_1 \cap \Ker\delta_2$, a contradiction.

Thus $a \mid b$. If $a \nmid c$, then $\KK[x_1, x_2, x_3]_c = \langle x_3\rangle$ and $x_1, x_3 \in \Ker\delta_1 \cap \Ker\delta_2$. Consequently $a, b, c$ are divided by $a$, whence $a = 1$. Then we have $\delta_i(x_2) \in \KK[x_1, x_2, x_3]_b = \langle x_2, x_1^b\rangle$, and by Lemma~\ref{lem0}
$$
\delta_i\colon
x_1 \mapsto 0, \quad
x_2 \mapsto \beta_i x_1^b,
\quad\quad i = 1, 2.$$

Let $c = bd + e$, $0 \le e < b$. Then $\KK[x_1, x_2, x_3]_c = \langle x_3\rangle \oplus \langle x_1^{c-bl} x_2^l \mid 0 \le l \le d\rangle$ and by Lemma~\ref{lem0}
$$
\delta_i\colon
x_1 \mapsto 0, \quad
x_2 \mapsto \beta_i x_1^b, \quad
x_3 \mapsto P^{(i)}_{b, c}(x_1, x_2),
\quad\quad i = 1, 2,$$
where $P^{(i)}_{b, c}$ is a polynomial of the form $\sum \limits_{l = 0}^d \xi_{i l} x_1^{c - bl} x_2^l$.

\medskip

{\bf Step 2.} One can see that the derivations $\delta_1$ and $\delta_2$ of form~\eqref{eqchange1} commute if and only if the condition $\beta_2\dfrac{\pa P^{(1)}_{b, c}}{\pa x_2}(x_1, x_2) = \beta_1\dfrac{\pa P^{(2)}_{b, c}}{\pa x_2}(x_1, x_2)$ holds. This implies that the corresponding coefficients of $\beta_2 P_{b, c}^{(1)}(x_1, x_2)$ and $\beta_1 P_{b, c}^{(2)}(x_1, x_2)$ are equal except of coefficients at~$x_1^c$, that is
$P_{b, c}^{(i)}(x_1, x_2) = \beta_i P_{b, c}(x_1, x_2) + \gamma_i x_1^c, \; i = 1, 2,$
for some $P_{b, c}(x_1, x_2) = \sum \limits_{l = 1}^d \xi_l x_1^{c - bl} x_2^l$. Thus
\begin{equation} \label{eqchange2}
\delta_i\colon
x_1 \mapsto 0, \quad
x_2 \mapsto \beta_i x_1^b, \quad
x_3 \mapsto \beta_i P_{b, c}(x_1, x_2) + \gamma_i x_1^c,
\quad\quad i = 1, 2.
\end{equation}

\medskip

{\bf Step 3.} If $d=1$ we may assume that $\xi_1=1$ or $\xi_1=0$, and derivations~(\ref{eqchange2}) already have the desired form. Assume that $d\ge 2$ and consider the homogeneous change of coordinates in $\AA^3$
$$
(\ti x_1, \ti x_2, \ti x_3) =
\biggl(x_1, \; x_2, \; \alpha\Bigl(x_3 - \sum\limits_{l = 1}^{d-1}\dfrac{\xi_l}{l+1}x_1^{c - b(l+1)}x_2^{l+1}\Bigr)\biggr),
$$
where $\alpha \ne 0$ is a non-zero parameter. Under this change the derivations $\delta_1$ and $\delta_2$ in~\eqref{eqchange2} become
\begin{gather} \label{eqchange3}
\begin{aligned}
&\delta_i\colon\;\;
\ti x_1 \;\mapsto\; 0
\\
&\ti x_2 \;\mapsto\; \beta_i x_1^b = \beta_i \ti x_1^b
\\
&\ti x_3 \;\mapsto\; \alpha\Bigl(\beta_i \sum \limits_{l = 1}^d \xi_l x_1^{c-bl}x_2^l + \gamma_i x_1^c
- \sum\limits_{l = 1}^{d-1}\dfrac{\xi_l}{l+1} x_1^{c-b(l+1)} (l+1)x_2^l \beta_i x_1^b\Bigr) =
\alpha\bigl(\beta_i\xi_d \ti x_1^e \ti x_2^d + \gamma_i \ti x_1^c\bigr)\quad
\end{aligned}\raisetag{3\baselineskip}
\end{gather}
with $i = 1, 2$. It remains to let $\alpha = 1/\xi_d$ if $\xi_d \ne 0$, and $\alpha = 1$ otherwise. The former case leads to derivations of Type 1, while in the latter case we obtain derivations of Type 2. This completes the proof of Proposition~\ref{prop2der}.
\end{proof}

\begin{proof}[Proof of Theorem~\ref{A3}]
Classification of monoids of rank 0, 2, and 3 follows from Proposition~\ref{corank01n}.

\smallskip

Consider a monoid of rank 1 without zero. By Lemma~\ref{l2l}, such a monoid is isomorphic
either to $S_0\times\GG_a^2$ with $\dim S_0=1$ or to $S_0\times \GG_a$ with $\dim S_0=2$, and the rank of $S_0$ equals~1 in both cases.
So we come to the monoid $M \underset{b}{+} A \underset{c}{+} A$ with $b=c=0$ in the first case
and with $b=0$, $c>0$ in the second case.

\smallskip

It remains to classify monoids of rank 1 with zero. This means that $G = \GG_m \times \GG_a^2$
and the action of $\GG_m$ on $\AA^3$ has a unique closed orbit. By Lemma~\ref{l3l}, we may assume that the action is linear. This action corresponds to a $\ZZ$-grading on $\KK[x_1, x_2, x_3]$ such that the variables $x_1,x_2,x_3$ are homogeneous and have like-sign degrees $a,b,c$.
We assume that $0<a \le b \le c$. Since the torus action is effective it follows that the numbers $a, b, c$ are coprime.

By Lemma~\ref{grLND}, we have two commuting locally nilpotent derivations $\delta_1, \delta_2$ of degree zero on $\KK[x_1, x_2, x_3]$ such that the corresponding action of $\GG_m \times \GG_a^2$ on $\AA^3$ is effective and has an open orbit.

If the dimension of $\Ker\delta_1 \cap \Ker\delta_2 \cap \langle x_1, x_2, x_3\rangle$ is at least $2$ then there are two non-proportional linear functions on $\AA^3$ remaining constant under the $\GG_a^2$-action, whence the $G$-action has no open orbit.

Thus according to Proposition~\ref{prop2der} we obtain that the derivations $\delta_i$ have the form
$$\delta_i\colon\;
x_1 \;\mapsto\; 0, \quad
x_2 \;\mapsto\; \beta_i x_1^b, \quad
x_3 \;\mapsto\; \gamma_i x_1^c+ \alpha\beta_i x_1^e x_2^d,
\quad\quad i = 1, 2,$$
where $\alpha=0$ for Type~2 and, multiplying the variable $x_3$ by a suitable scalar we may assume that $\alpha=d+1$ for Type~1.

Let us find the corresponding action of the group $G$. The action of $\GG_a^2$ corresponds to locally nilpotent derivations in the Lie algebra generated by $\delta_1$ and $\delta_2$.
Since $G$ has an open orbit it follows that the tuples $(\beta_1, \beta_2)$ and $(\gamma_1, \gamma_2)$ are non-proportional. One can make a linear change of $\delta_1, \delta_2$ and assume that
$$\begin{aligned}
\delta_1&\colon
x_1 \;\mapsto\; 0, \quad
x_2 \;\mapsto\; x_1^b, \quad
x_3 \;\mapsto\; \alpha x_1^e x_2^d,
\\
\delta_2&\colon
x_1 \;\mapsto\; 0, \quad
x_2 \;\mapsto\; 0, \quad
x_3 \;\mapsto\; x_1^c.
\end{aligned}$$

Let $(t, \alpha_1, \alpha_2) \in G$. One can check that the $G$-action is given by
\begin{equation} \label{A3act}
\begin{aligned}
t \exp(\alpha_1\delta_1 + \alpha_2\delta_2)\colon \;
x_1 & \mapsto tx_1,\\
x_2 & \mapsto t^b \left[x_2 + \alpha_1 x_1^b \right], \\
x_3 & \mapsto t^c \Bigl[x_3 + \alpha_2 x_1^c +
\dfrac{\alpha}{d+1} \sum_{k = 1}^{d+1} \binom{d+1}{k} \alpha_1^k x_1^{e + b(k - 1)} x_2^{d-k+1}
\Bigr].
\end{aligned}
\end{equation}
It has an open orbit $\{x_1 \ne 0\}$. Then a $G$-embedding is given by the identification
$$\begin{aligned}
\GG_m \times \GG_a^2 & \longleftrightarrow \{x_1 \ne 0\} \subset \AA^3, \\
(t, \alpha_1, \alpha_2) & \longleftrightarrow (y_1, y_2, y_3) = (t, \alpha_1, \alpha_2) \circ (1, 0, 0) = \Bigl(t, \; t^b\alpha_1, \; t^c \Bigl[
\alpha_2 + \dfrac{\alpha}{d+1} \alpha_1^{d+1}\Bigr]\Bigl).
\end{aligned}$$
Substituting
$$
t = y_1, \quad \alpha_1 = y_1^{-b} y_2, \quad
\alpha_2 = y_1^{-c} y_3 - \dfrac{\alpha}{d+1}y_1^{-b(d+1)}y_2^{d+1}
$$
in~\eqref{A3act} we obtain a multiplication in $\AA^3$:
$$\begin{aligned}
&(y_1, y_2, y_3) * (x_1, x_2, x_3) =
\biggl(y_1x_1, \;\; y_1^b(x_2 + y_1^{-b} y_2 x_1^b), \\
&y_1^c \Bigl[x_3 + \Bigl(y_1^{-c}y_3 - \dfrac{\alpha}{d+1}y_1^{-b(d+1)} y_2^{d+1}\Bigr) x_1^c +
\dfrac{\alpha}{d+1} \sum_{k = 1}^{d+1} \binom{d+1}{k} y_1^{-bk} y_2^k x_1^{e + b(k - 1)} x_2^{d-k+1}
\Bigr]\biggr) =
\\
&= \Bigl(x_1y_1, \;\; x_1^by_2+y_1^bx_2, \;\;
x_1^c y_3 + y_1^c x_3 +
\dfrac{\alpha}{d+1}Q_{b, c}(x_1, y_1, x_2, y_2)\Bigl).
\end{aligned}$$
So with $\alpha=d+1$ we come to the monoid $M \underset{b}{+} A \underset{b,c}{+} A$, where
$1\le b\le c$, and with $\alpha=0$ we obtain the monoid $M \underset{b}{+} A \underset{c}{+} A$, where $1\le b\le c$.

\smallskip

Let us show that every two monoids of different types or of the same type with different values of parameters in Theorem~\ref{A3} are non-isomorphic. We may assume that these monoids have the same rank. Monoids of rank~0 and of rank~3 are unique, and the case of rank~2 follows from Proposition~\ref{corank01n}.

It remains to consider monoids of rank~1. If they do not have zero, they can be distinguished by the structures of the base and the fiber of the quotient morphism
$S\to S/\!/T$.

Assume now that both monoids have zero. Here the positive integers~$b, c$ are uniquely determined by the structure of the graded algebra on $\KK[x_1,x_2,x_3]$ defined by the
$T$-action on $S$. Thus we have to prove that the monoids $M \underset{b}{+} A \underset{b,c}{+} A$ and $M \underset{b}{+} A \underset{c}{+} A$ for the same $b, c$ are non-isomorphic. By~\eqref{A3act}, these monoids correspond to actions~\eqref{A3act_1} and~\eqref{A3act_2} respectively:
\begin{align} \label{A3act_1}
&\begin{aligned}
\phi_1(t, \alpha_1, \alpha_2)\colon \;
x_1 & \mapsto tx_1,\\
x_2 & \mapsto t^b \left[x_2 + \alpha_1 x_1^b \right], \\
x_3 & \mapsto t^c \Bigl[x_3 + \alpha_2 x_1^c + \sum_{k = 1}^{d+1} \binom{d+1}{k} \alpha_1^k x_1^{e + b(k - 1)} x_2^{d-k+1}
\Bigr];
\end{aligned}
\\
\label{A3act_2}
&\begin{aligned}
\phi_2(t, \alpha_1, \alpha_2)\colon \;
x_1 & \mapsto tx_1,\\
x_2 & \mapsto t^b \left[x_2 + \alpha_1 x_1^b \right], \\
x_3 & \mapsto t^c \Bigl[x_3 + \alpha_2 x_1^c\Bigr].
\end{aligned}
\end{align}
Consider the homogeneous component $\KK[x_1, x_2, x_3]_c = \langle x_3\rangle \oplus \langle x_1^{c-bl} x_2^l \mid 0 \le l \le d\rangle$ and the action of the unipotent part $\GG_a^2$ of the group $G(S)$ on it. Let us prove that for action~\eqref{A3act_1} there exists a unique one-parameter subgroup in $\GG_a^2$ which is a stabilizer of an element of $\KK[x_1, x_2, x_3]_c$, while for action~\eqref{A3act_2} there are infinitely many such one-parameter subgroups.

Let $f = \lambda x_3 + f_0 = \lambda x_3 + \sum\limits_{l = 0}^d \lambda_l x_1^{c-bl} x_2^l \in \KK[x_1, x_2, x_3]_c$.
Then $\phi_1(1, -\alpha_1, -\alpha_2)(f)$ equals
\[
\lambda\Bigl[x_3 + \alpha_2 x_1^c + \sum_{k = 1}^{d+1} \binom{d+1}{k} \alpha_1^k x_1^{e + b(k - 1)} x_2^{d-k+1}
\Bigr] + \sum_{l = 0}^d \lambda_l x_1^{c - bl} (x_2 + \alpha_1x_1^b)^l.
\]
Equating coefficients in the condition $\phi_1(1, -\alpha_1, -\alpha_2)(f) = f$ at $x_1^{c-bm} x_2^m$, $0 \le m \le d$, we obtain the following system of linear equations in $\lambda$ and $\lambda_l$, $1 \le l \le d$:
\[
\begin{pmatrix}
\alpha_1 & \alpha_1^2 & \alpha_1^3 & \hdotsfor{1} & \quad\alpha_1^{d\phantom{+0}} & \alpha_2 + \alpha_1^{d+1}
\\
0\; & \binom{2}{1} \alpha_1 & \binom{3}{1} \alpha_1^2 & \hdotsfor{1} & \binom{d}{1} \alpha_1^{d-1} & \binom{d+1}{1} \alpha_1^{d\phantom{+0}}
\\
0\; & \quad0\quad & \;\binom{3}{2} \alpha_1 & \hdotsfor{1} & \binom{d}{2} \alpha_1^{d-2} & \binom{d+1}{2} \alpha_1^{d-1}
\\
\hdotsfor{6}
\\
0\; & \hdotsfor{2} & \quad\;0 & \binom{d}{d-1}\alpha_1 & \binom{d+1}{d-1} \alpha_1^{2\phantom{+0}}
\\
0\; & \hdotsfor{2} & \quad\;0 & \quad0\quad\;\; & \binom{d+1}{d} \alpha_1^{\phantom{1+0}} &
\end{pmatrix}
\begin{pmatrix}
\lambda_1 \\ \lambda_2 \\ \lambda_3 \\ \ldots \\ \lambda_d \\ \lambda
\end{pmatrix}
=
\begin{pmatrix}
0 \\ 0 \\ 0 \\ \ldots \\ 0 \\ 0	
\end{pmatrix}.
\]
Here the coefficient $\lambda_0$ is arbitrary and we can assume that $\lambda_0 = 0$.
We claim that the stabilizer of $f$ in the unipotent part $\GG_a^2$ of the group $G(S)$ consists of elements $(\alpha_1,\alpha_2)$ with
\begin{enumerate}
\item \label{en0} $\alpha_1 = \alpha_2 = 0$ if $\lambda \ne 0$; 
\item \label{enK} $\alpha_1 = 0$ if $\lambda = 0$, $f_0 \ne 0$; 
\item \label{enK2} arbitrary $\alpha_1, \alpha_2$ if $\lambda = 0$, $f_0 = 0$; 
\end{enumerate}

Indeed, the condition $\lambda \ne 0$ with the last equation of the system implies $\alpha_1 = 0$. Then the first equation yields $\lambda\alpha_2 = 0$, whence $\alpha_2 = 0$.
In the case $\lambda = 0$ we get a system of $n-1$ linear equations in $\lambda_l$, $1 \le l \le d$. Since at least one $\lambda_l$ is non-zero, the corresponding leading principal minor equals $0$, whence $\alpha_1 = 0$.

Thus the stabilizer for action~\eqref{A3act_1} is one-parameter whenever it is equal to $\{\alpha_1 = 0\}$.

\medskip

For action~\eqref{A3act_2}, the polynomials $x_3+\lambda_1x_1^{c-b}x_2$ have unipotent one-parameter stabilizers $\{\alpha_2+\lambda_1\alpha_1 = 0\}$ that are non-equal for different $\lambda_1$. Consequently, the monoids $M \underset{b}{+} A \underset{b,c}{+} A$ and $M \underset{b}{+} A \underset{c}{+} A$ are non-equivalent.
This completes the proof of Theorem~\ref{A3}.
\end{proof}

\begin{remark}
Under the assumption that $b$ divides $c$ for $M \underset{b}{+} A \underset{c}{+} A$ the action of the unipotent part $\GG_a^2$ of $G(S)$ on $\AA^3\setminus G(S)$ is trivial, while for $M \underset{b}{+} A \underset{b,c}{+} A$ it is not.
\end{remark}


\section{Bilinear monoids and finite-dimensional algebras}
\label{sec5}
We say that an algebraic monoid $S$ on $\AA^n$ is bilinear if the multiplication
$$
\mu\colon S\times S\to S, \quad \mu(a,b)=ab
$$
is a bilinear map in some coordinates on $\AA^n$. In this case the space $S$ has a structure of a finite-dimensional associative algebra with unit. Conversely, for every finite-dimensional associative  algebra $A$ with unit the multiplicative monoid $(A,*)$ is a bilinear monoid.

\begin{lemma} \label{l16l}
Let $A_1$ and $A_2$ be finite dimensional associative algebras with unit. Assume that the multiplicative monoids $(A_1,*)$ and $(A_2,*)$ are isomorphic. Then the algebras $A_1$ and $A_2$ are isomorphic as well.
\end{lemma}

\begin{proof}
Every bilinear monoid has zero. So we identify both $A_1$ and $A_2$ with $\AA^n$ sending zeroes to the origin. The first bilinear operation is given by
$$
x*y=\bigl(\sum\gamma^{(1)}_{ij}x_iy_j,\ldots,\sum\gamma^{(n)}_{ij}x_iy_j\bigr)
$$
with some $\gamma^{(k)}_{ij}\in\KK$, while the second one is $f^{-1}(f(x)*f(y))$
for some polynomial automorphism $f\colon\AA^n\to\AA^n$ with $f(0)=0$. Every automorphism
$f$ can be represented by $n$-tuple of polynomials $(f_1(x),\ldots,f_n(x))$. Let us replace $f$
by a linear automorphism $h$, where $h_i(x)$ is equal to the linear part of $f_i(x)$ for
all $i=1,\ldots,n$. Since the second multiplication is bilinear, the multiplications
$f^{-1}(f(x)*f(y))$ and $h^{-1}(h(x)*h(y))$ coincide. Hence the automorphism $h$ establishes
an isomorphism between the algebras $A_1$ and $A_2$.
\end{proof}

Let us give a characterization of bilinear monoids in terms of group embeddings.

\begin{proposition} \label{propmul}
Let $S$ be an algebraic monoid on the affine space $\AA^n$. Then $S$ is bilinear
if and only if the action of the group $G(S)\times G(S)$ on $\AA^n$ by left and right multiplication is linearizable.
\end{proposition}

\begin{proof}
Assume that the group $G(S)\times G(S)$ acts on $\AA^n$ linearly. The multiplication map $\AA^n\times\AA^n\to\AA^n$ is given by the comorphism $\KK[\AA^n]\to\KK[\AA^n]\otimes\KK[\AA^n]$. Since the action of $G(S)\times G(S)$ on $\AA^n$ is linear, for the restriction of the comorphism to the subspace $(\AA^n)^*\subseteq\KK[\AA^n]$ of all linear functions on $\AA^n$ we have
$$
(\AA^n)^*\to \KK[G]\otimes (\AA^n)^* \quad \text{and} \quad (\AA^n)^*\to (\AA^n)^*\otimes\KK[G].
$$
So the image of $(\AA^n)^*$ is contained in the intersection $(\KK[G]\otimes (\AA^n)^*)\cap((\AA^n)^*\otimes\KK[G])=(\AA^n)^*\otimes (\AA^n)^*$. Hence the multiplication on $\AA^n$ is given by the linear map $\AA^n \otimes \AA^n\to \AA^n$ dual to $(\AA^n)^*\to (\AA^n)^*\otimes (\AA^n)^*$. This proves that the monoid $S$ is bilinear.

The converse implication is straightforward.
\end{proof}

Since every action with an open orbit of a reductive group on an affine space is linearizable~\cite{Lu}, we conclude that every reductive monoid on $\AA^n$ is bilinear. Clearly, the multiplicative monoid of a finite dimensional algebra $A$ is reductive if and only if $A$ is a semisimple $A$-module. By the Artin-Wedderburn Theorem, this is the case if and only if $A$ is a direct sum of matrix algebras $\text{Mat}(n_i\times n_i,\KK)$.

We know that monoids with an affine algebraic group $G$ as the group of invertible elements are precisely the embeddings of $G$ into affine varieties. In particular, for bilinear monoids on affine spaces we obtain a prehomogeneous $G$-module structure on $\AA^n$ with trivial generic stabilizer. Since for a commutative group $G$ there is no difference between left and right multiplications, here the inverse implication holds. Namely, every faithful prehomogeneous module of a commutative group $G$ carries a structure of a commutative algebra with unit. The following example shows that already for solvable groups this is not the case.

\begin{example}
Let us consider the group
$$
G=\left\{
\begin{pmatrix}
t & a \\
0 & t^{-1}
\end{pmatrix}, \
t\in\GG_m, \ a\in\GG_a\right\}
$$
and its tautological module $\KK^2$. The orbit  of the vector $(0,1)$ is open in $\KK^2$, it consists of the vectors $(a,t^{-1})$ or, equivalently, of the vectors $(x,y)$, $y\ne 0$. The right multiplication by an element
$$
\begin{pmatrix}
s & b \\
0 & s^{-1}
\end{pmatrix}^{-1}
$$
gives the vector $(sa-tb,t^{-1}s)$ or, equivalently, $(sx-by^{-1},sy)$. Such an action can not be extended to $\KK^2$.
\end{example}

It is worth noting that bilinear monoids define group embeddings into projective spaces. Namely, let $A$ be a finite-dimensional associative algebra with unit and $C$ be the subgroup of nonzero scalars in $G(A)$. Then we have an open embedding $G(A)/C\hookrightarrow \PP(A)$. By~\cite[Proposition~5.1]{KL}, all open equivariant embeddings of commutative affine algebraic groups into projective spaces appear this way.

Let us consider commutative bilinear monoids in more details. It is well known that every $n$-dimensional commutative associative algebra $A$ with unit admits a decomposition $A=A_1\oplus\ldots\oplus A_r$ into a direct sum of local algebras $A_i$ with maximal ideals $\mm_i$, see~\cite[Theorem~8.7]{AM}. Such a decomposition is unique up to order of summands. Moreover, every algebra $A_i$ decomposes as a vector space to $\KK\oplus\mm_i$, all elements in $\mm_i$ are nilpotent and all elements in $A_i\setminus\mm_i$ are invertible. In particular, the group of invertible elements $G(A_i)$ equals $\KK^{\times}\oplus\mm_i$. It is a connected commutative affine algebraic group of rank~1; an isomorphism  $\KK^{\times}\oplus\mm\to \GG_m\times\GG_a^s$ is given by $(\lambda, x)\mapsto (\lambda,\ln(1+\lambda^{-1}x))$. In general, the group $G(A)$ is isomorphic to $\GG_m^r\times\GG_a^{n-r}$.

\begin{example}
There are precisely four 3-dimensional commutative algebras, namely
$$
\KK\oplus\KK\oplus\KK,\quad \KK\oplus\KK[T_1]/(T_1^2),\quad \KK[T_1,T_2]/(T_1^2,T_1T_2,T_2^2),\quad \KK[T_1]/(T_1^3).
$$
They give rise to the bilinear multiplications
$$
(x_1y_1,x_2y_2, x_3y_3), \quad (x_1y_1,x_2y_2, x_2y_3+x_3y_2), \quad (x_1y_1, x_1y_2+x_2y_1, x_1y_3+x_3y_1),
$$
$$
(x_1y_1, x_1y_2+x_2y_1, x_1y_3+x_2y_2+x_3y_1).
$$
In notation of Theorem~\ref{A3} these are the monoids
$$
3M,\quad M + M \underset{0, 1}{+} A,\quad M \underset{1}{+} A \underset{1}{+} A,\quad
M \underset{1}{+} A \underset{1,1}{+} A.
$$
\end{example}

By construction, the orbits of the group $G(A)$ on the monoid $(A,*)$ correspond to principal ideals of the algebra $A$. In particular, already for the monoid $M \underset{1}{+} A \underset{1}{+} A$ on $\AA^3$ the number of $G(A)$-orbits on $A$ is infinite. More precisely, for a local algebra $A$ the number of $G(A)$-orbits on $A$ is finite if and only if $A$ is $\KK[T_1]/(T_1^n)$, cf.~\cite[Proposition~3.7]{HT}.

A classification of local algebras is known up to dimension 6; see \cite[Section~2]{Ma} or \cite[Section~3]{HT}. The table below represents the numbers of isomorphy classes of local algebras of small dimensions:

\[\begin{array}{c|c|c|c|c|c|c|c}
  \dim A & 1 & 2 & 3 & 4 & 5 & 6 & \ge7 \\\hline
  \#A & 1 & 1 & 2 & 4 & 9 & 25 & \infty
\end{array}\]

Starting from dimension~7, such algebras are parameterized by moduli spaces of positive dimension \cite[Example~3.6]{HT}. Applying these results and Lemma~\ref{l16l} we come to the following observation.

\begin{proposition} \label{pinfp}
There are pairwise non-isomorphic commutative monoid structures on the affine space $\AA^n$ with $n\ge 7$ parameterized by moduli spaces of positive dimension.
\end{proposition}


\section{Cox rings, toric varieties, and additive actions}
\label{sec6}

In this section we establish a connection between commutative algebraic monoids on affine spaces and actions with an open orbit of commutative unipotent groups on toric varieties.

Let $X$ be an irreducible algebraic variety over the ground field $\KK$. An additive action on $X$ is a regular faithful action $\GG_a^s\times X\to X$ with an open orbit. Let us recall that a variety $X$ is toric if $X$ is normal and there exists an action of an algebraic torus on $X$ with an open orbit. Additive actions on toric varieties are studied in~\cite{AR}.

If a variety $X$ admits an additive action then every regular invertible function on $X$ is constant and the divisor class group $\Cl(X)$ is a free finitely generated abelian group. For a toric variety $X$ the latter conditions imply that $X$ can be realized as a good quotient $\pi\colon U\stackrel{/\!/T}\longrightarrow X$ of an open subset $U\subseteq\AA^n$ whose complement is a collection of coordinate subspaces of codimensions at least 2 by a linear action of a torus $T$. Such a realization can be chosen in a canonical way. Namely, the Cox ring
$$
R(X)=\bigoplus_{[D]\in\Cl(X)} H^0(X,D)
$$
of a toric variety $X$ is a polynomial ring graded by the group $\Cl(X)$. The grading defines a linear action of the characteristic torus $T:=\Spec(\KK[\Cl(X)])$ on the total coordinate space $\AA^n:=\Spec(R(X))$. A canonically defined open subset $U\subseteq\AA^n$ whose complement is a union of some coordinate subspaces of codimensions at least 2, gives rise to the so-called characteristic space  $p\colon U\stackrel{/\!/T}\longrightarrow X$; we refer to~\cite{Cox} and \cite[Chapter~II]{ADHL} for details.

An additive action $\GG_a^s\times X\to X$ can be lifted to an action $\GG_a^s\times\AA^n\to\AA^n$ on the total coordinate space commuting with the $T$-action~\cite[Theorem~4.2.3.2]{ADHL}. This defines a faithful action $G\times\AA^n\to\AA^n$ of the commutative group $G:=T\times\GG_a^s$ with an open orbit. Let us say that the action $G\times\AA^n\to\AA^n$ is associated with the given additive action on a toric variety $X$.

Further, faithful actions $G\times\AA^n\to\AA^n$ with an open orbit are in bijection with commutative monoids on $\AA^n$ having $G$ as the group of invertible elements. So we obtain a commutative monoid on $\AA^n$ associated with the given additive action on a toric variety $X$.

\begin{example}
Consider the action $\GG_a^n\times\AA^n\to\AA^n$ by translations. This is an additive action on a toric variety $X = \AA^n$, the Cox ring $R(X)$ coincides with $\KK[X]$, the torus $T$ is trivial, and the variety $\Spec(R(X))$ coincides with $\AA^n$. So we come to the monoid of rank~0, see Proposition~\ref{corank01n},1).
\end{example}

Let us say that a toric variety $X$ is 2-complete if $X$ can not be realized as an open toric subset of a toric variety $X'$ such that $\codim_{X'} X'\setminus X\ge 2$. It is easy to see that every toric variety $X''$ can be realized as an open subset of a 2-complete toric variety $X$ with $\codim_{X} X\setminus X''\ge 2$, the varieties $X''$ and $X$ share the same (graded) Cox ring and every additive action on $X''$ can be extended to $X$. So from now on we consider commutative monoids associated only with additive actions on 2-complete toric varieties.

\begin{proposition} \label{assact}
Commutative bilinear monoids without direct factors $(\AA^1,*)$ are precisely the monoids associated with additive actions on products of projective spaces.
\end{proposition}

\begin{proof}
Let $S$ be a commutative bilinear monoid on $\AA^n$. Then $\AA^n$ has a structure of a commutative algebra, and there is a decomposition $\AA^n=A_1\oplus\ldots\oplus A_r$ into a direct sum of local algebras. Each element $(t_1,\ldots,t_r)$ of the maximal torus $T$ in $G(S)$ acts on every subspace $A_i$ via scalar multiplication by $t_i$. Let $d_i:=\dim A_i$. Then the torus $T$ acts on $\AA^n$ linearly with characters $e_1$ ($d_1$ times),\ldots, $e_r$ ($d_r$ times), where $e_1,\ldots,e_r$ form a basis of the lattice of characters of the torus $T$. The condition that $S$ contains no direct factors $(\AA^1,*)$ means that all $d_i\ge 2$.
It is easy to show (see~\cite[Exercise~2.13]{ADHL}) that there is a unique maximal open subset $U\subseteq\AA^n$ such that there exists a good quotient $\pi\colon U\stackrel{/\!/T}\longrightarrow X$ which is the characteristic space of $X$; namely, $U=(A_1\setminus\{0\})\times\ldots\times(A_r\setminus\{0\})$ and $X=\PP(A_1)\times\ldots\times\PP(A_r)$.

Conversely, consider an additive action $\GG_a^s\times X\to X$ with $X=\PP(V_1)\times\ldots\times\PP(V_r)$, where $V_i$ are some vector spaces of dimension at least two. The Picard group of $X$ is freely generated by the line bundles $L_1,\ldots,L_r$ corresponding to ample generators of the Picard groups of the factors $\PP(V_1),\ldots,\PP(V_r)$. The space of global sections of $L_i$ is identified with the dual space $V_i^*$. By~\cite[Section~2.4]{KKLV}, every line bundle $L_i$ admits a $\GG_a^s$-linearization. Thus the lifted action of the group $G=T\times\GG_a^s$ to the total coordinate space $\AA^n=V_1\oplus\ldots\oplus V_r$ of $X$ is linear and faithful, and we obtain the desired bilinear monoid structure on $\AA^n$.
\end{proof}

The results of Section~\ref{sec4} allow to classify additive actions on a weighted projective plane $\PP(a,b,c)$. Without loss of generality we may assume that $0<a\le b\le c$ and $a,b,c$ are pairwise coprime. By~\cite[Proposition~2]{AR}, an additive action on $\PP(a,b,c)$ exists if and only if $a=1$. The Cox ring $R(X)$ of $X=\PP(a,b,c)$ is the polynomial ring $\KK[x_1,x_2,x_3]$ with a $\ZZ$-grading given by $\deg(x_1)=a$, $\deg(x_2)=b$, $\deg(x_3)=c$. An additive action on $X$ corresponds to a pair of commuting homogeneous locally nilpotent derivations on $\KK[x_1,x_2,x_3]$ of degree zero. Thus Proposition~\ref{prop2der} and formulas~(\ref{A3act_1}),~(\ref{A3act_2}) imply the following result.

\begin{proposition} \label{WPP}
Let $\PP(1,b,c)$ be a weighted projective plane with $1\le b\le c$, $(b,c)=1$. There are exactly two non-isomorphic additive actions $\GG_a^2\times\PP(1,b,c)\to \PP(1,b,c)$ given by
$$
(\alpha_1,\alpha_2)\circ (x_1,x_2,x_3)=(x_1, x_2+\alpha_1 x_1^b,x_3 + \alpha_2 x_1^c)
$$
and
$$
(\alpha_1,\alpha_2)\circ (x_1,x_2,x_3)=\Bigl(x_1, x_2+\alpha_1 x_1^b,x_3 + \alpha_2 x_1^c +\sum_{k = 1}^{d+1} \binom{d+1}{k} \alpha_1^k x_1^{e + b(k - 1)} x_2^{d-k+1}\Bigr),
$$
where $c = b d + e, \; d, e \in \ZZ, \; 0 \le e < b$.
\end{proposition}

The last example of this section provides some commutative monoids on $\AA^4$.

\begin{example}\label{exref}
Let $X$ be the Hirzebruch surface $\FF_d$, $d\in\ZZ_{\ge 0}$. This is a toric variety given by a complete fan with rays generated by the vectors $(0, 1), (1, 0), (0, -1), (-1, d)$. The Cox ring of $\FF_d$ is the polynomial ring $\KK[x_1, x_2, x_3, x_4]$ with a $\ZZ^2$-grading given by
$$
\deg x_1 = (1, 0), \quad \deg x_2 = (0, 1), \quad \deg x_3 = (1, 0), \quad \deg x_4 = (d, 1).
$$
In this case the open subset $U\subseteq\AA^4$ in the quotient presentation of $\FF_d$ is $$
\AA^4 \setminus (\{x_1=x_3=0\} \cup \{x_2=x_4=0\}).
$$
By~\cite[Proposition~5.5]{HT}, the variety $\FF_d$ with $d>0$ admits two additive actions $\GG_a^2\times\FF_d\to\FF_d$, one is normalized by the 2-torus acting on $\FF_d$, and the another one is not. For $d=0$, we have only a normalized additive action on $\FF_0=\PP^1\times\PP^1$.

Explicitly, the lifting of the normalized action to $\AA^4$ is given by
$$
(t_1, t_2, \alpha_1, \alpha_2) \circ (x_1, x_2, x_3, x_4) = \bigl(t_1x_1, \; t_2x_2, \; t_1(x_3 + \alpha_1x_1), \; t_1^dt_2(x_4+\alpha_2x_1^dx_2)\bigr),
$$
while the non-normalized action lifts as
\begin{multline*}
(t_1, t_2, \alpha_1, \alpha_2) \circ (x_1, x_2, x_3, x_4) =\\= \Bigl(t_1x_1, \; t_2x_2, \; t_1(x_3 + \alpha_1x_1), \; t_1^dt_2\Bigl(x_4+\Bigl(\alpha_2 + \frac{\alpha_1^2}{2}\Bigr)x_1^dx_2 + \alpha_1x_1^{d-1}x_2x_3\Bigr)\Bigr),
\end{multline*}
cf.~\cite[Example~6.4]{AR}. The corresponding commutative monoids of rank~2 on $\AA^4$ are
$$
x*y = (x_1y_1, \; x_2y_2, \; x_1y_3 + y_1x_3, \; x_1^dx_2y_4 + y_1^dy_2x_4)
$$
and
$$
x*y = (x_1y_1, \; x_2y_2, \; x_1y_3 + y_1x_3, \; x_1^dx_2y_4 + y_1^dy_2x_4 + x_1^{d-1}y_1^{d-1}x_2y_2x_3y_3).
$$
\end{example}

\section{Idempotent elements in commutative monoids}
\label{sec7}

In this section we illustrate the well-known structure theory of affine algebraic monoids in the case of commutative monoids. This case is much simpler, and for convenience of the reader we provide short proofs whenever it is possible.

An element $e\in S$ of an algebraic monoid $S$ is an idempotent if $e^2=e$. We denote by $E(S)$ the set of all idempotents of $S$. Let us take a closer look at the set $E(S)$ for a commutative monoid $S$.

\begin{lemma}
Every $G(S)$-orbit on $S$ contains at most one idempotent.
\end{lemma}

\begin{proof}
If $e,ge \in E(S)$ with $g\in G(S)$, then $ge=(ge)^2=g^2e^2=g^2e$ and $e=ge$.
\end{proof}

\begin{proposition} \label{p3cp}
Let $S$ be a commutative algebraic monoid with multiplication $\mu\colon S\times S\to S$ and
$\mathcal{O}$ be a $G(S)$-orbit on $S$. The following conditions are equivalent.
\begin{enumerate}
\item
The orbit $\mathcal{O}$ contains an idempotent.
\item
The intersection $\mu(\mathcal{O}\times\mathcal{O})\cap\mathcal{O}$ is non-empty.
\item
We have $\mu(\mathcal{O}\times\mathcal{O})=\mathcal{O}$ and the restriction
$$
\mu|_{\mathcal{O}\times\mathcal{O}}\colon\mathcal{O}\times\mathcal{O}\to\mathcal{O}
$$
defines a group structure on $\mathcal{O}$.
\end{enumerate}
\end{proposition}

\begin{proof}
$(1)\Rightarrow (2)$ If $e\in\mathcal{O}$ is an idempotent then $\mu(e,e)=e\in\mathcal{O}$.

\smallskip

$(2)\Rightarrow (3)$ If $a,b\in\mathcal{O}$ and $ab\in\mathcal{O}$ then
$$
(ga)(g'b)=(gg')(ab)\in\mathcal{O}
$$
for any $g,g'\in G(S)$. This implies $\mu(\mathcal{O}\times\mathcal{O})=\mathcal{O}$.

Now suppose that $a^2=ga$ for some $g\in G(S)$. Let $e=g^{-1}a$. Then $e^2=e$ and for any elements $ge$ and $g'e$ in $\mathcal{O}$ we have $geg'e=gg'e^2=gg'e$. This shows that $\mathcal{O}$ is identified with the factor group $G(S)/H$, where $H$ is the stabilizer of $e$ in $G(S)$.

\smallskip

$(3)\Rightarrow (1)$ The unit element of the group $\mathcal{O}$ is the desired idempotent.
\end{proof}

Let us introduce some classes of elements in commutative monoids. We say that
an element $a\in S$ is group-like if the orbit $G(S)a$ contains an idempotent.
For a monoid $S$ with zero we call an element $a\in S$ a zero-divisor
if there is a nonzero $b\in S$ such that $ab=0$. Finally, an element $a\in S$
is nilpotent if there is a positive integer $m$ such that $a^m=0$.
These properties are constant along $G(S)$-orbits. Moreover, a nilpotent is a group-like element if and only if it is zero.

\begin{example}
Let $S$ be a monoid of corank~1 as in Proposition~\ref{corank01n}. Assume that all non-negative integers $b_i$ defining this monoid are positive. This means that $S$ does not have the monoid $(\AA^1,*)$ as a direct factor. Then non-invertible group-like elements in $S$
are precisely the elements $(x_1,\ldots,x_{n-1},x_n)$ with $x_n=0$ and at least one more zero coordinate. The set $E(S)$ consists of the elements $(\epsilon_1,\ldots,\epsilon_{n-1},0)$ with $\epsilon_i=0,1$, and the elements $(0,\ldots,0,x_n)$ are nilpotent.
\end{example}

\begin{lemma}
Let $S$ be an algebraic monoid with zero. Then every element in $S$ is either invertible or a zero-divisor.
\end{lemma}

\begin{proof}
Assume that an element $a\in S$ is not invertible. Then the image of the morphism $S\to S$, $s\mapsto as$ is contained in $S\setminus G(S)$. Since $\dim(S\setminus G(S))<\dim S$, the preimage of the point $0$ in $S$ can not consist of one element~$0$; see~\cite[Theorem~4.1]{Hum}.
\end{proof}

The next result is a very particular case of strong $\pi$-regularity for algebraic semigroups~\cite{BR}.

\begin{proposition}\label{q1q}
Let $S$ be a commutative algebraic monoid. Then there is a positive integer $m$ such that for every $a\in S$ the element $a^m$ is group-like.
\end{proposition}

\begin{proof}
Let $\mathcal{O}$ be the $G(S)$-orbit of an element $a\in S$. Since $\mu(G(S),\mathcal{O})=\mathcal{O}$ and the group $G(S)$ is dense in $S$,
we see that $\mu(S,\mathcal{O})$ is contained in the closure $\overline{\mathcal{O}}$.
In particular, $\mu(\mathcal{O},\mathcal{O})$ is contained in $\overline{\mathcal{O}}$.
If the set $\mu(\mathcal{O},\mathcal{O})$ intersects $\mathcal{O}$, the element $a$ is group-like by Proposition~\ref{p3cp}. Otherwise the element $a^2$ lies in a $G(S)$-orbit of smaller dimension.

By the same arguments, either the element $a^2$ is group-like or $a^4$ lies in an orbit of even smaller dimension. Since the dimension is a non-negative integer, we find a number $m$ such that $a^m$ is a group-like element. By construction, we have $m\le 2^{\dim S}$. Since any power
of a group-like element is again group-like, we can find one value $m$ for all elements $a\in S$.
\end{proof}

\begin{example}
Let $S$ be a monoid of corank~1 as in Proposition~\ref{corank01n}. Then for every non-invertible element $a\in S$ the last coordinate of the element $a^2$ is zero, hence $a^2$ is group-like.
\end{example}

Let $T$ be the maximal torus in $G(S)$ for a commutative affine monoid $S$ and $\chi_1\ldots,\chi_l$ be the weights of some $T$-semi-invariant generators of the algebra $\KK[S]$. Consider the rational vector space $M_{\QQ}$ spanned by the characters of $T$. Let $C(S)$ be the cone in $M_{\QQ}$ generated by $\chi_1\ldots,\chi_l$. The cone $C(S)$ has dimension $r$.

Assume for a moment that $G(S)=T$. We call such monoids toric. Faces of the cone $C(S)$ are in bijection with $T$-orbits on $S$. Moreover, every $T$-orbit on $S$ contains a unique idempotent: under the closed embedding of $S$ in $\AA^l$ idempotents are precisely the points in $S$ with $(0,1)$-coordinates; see~\cite{Neeb} for details. In particular, in toric monoids all elements are group-like.

By a result of Putcha, for an affine monoid $S$ the conjugacy class of every idempotent intersects the submonoid $\overline{T}$ \cite[Theorem~6]{LLC}. If $S$ is commutative, we have $E(S)=E(\overline{T})$.

\begin{proposition} \label{us}
Let $S$ be a commutative monoid of rank $r$ on $\AA^n$. Then
\begin{enumerate}
\item
$S$ contains a finite number of idempotents and this number is at least $2^r$;
\item
the number of idempotents is $2^r$ provided $r\le 2$.
\end{enumerate}
\end{proposition}

\begin{proof}
For a commutative monoid $S$ and its toric submonoid $\overline{T}$ the cones $C(S)$ and $C(\overline{T})$ coincide. Since the sets $E(S)$ and $E(\overline{T})$ coincide as well, the number of idempotents in $S$ equals the number of faces in $C(S)$. Lemma~\ref{l2l} shows that $C(S)$ is a pointed cone of dimension~$r$, so it contains at least $2^r$ faces. Finally, for $r\le 2$ the number of faces is $2^r$.
\end{proof}

\begin{remark}
We do not have an example of a commutative monoid of rank $r$ on $\AA^n$ with more than $2^r$ idempotents.
\end{remark}

The following observation generalizes a basic fact on finite-dimensional local algebras.

\begin{corollary} \label{q3q}
Let $S$ be a commutative algebraic monoid with zero of rank 1 on $\AA^n$.
Then any element of $S$ is either invertible or nilpotent.
\end{corollary}

\begin{proof}
By Proposition~\ref{us}, the only idempotents in $S$ are $1$ and $0$. The assertion follows from Proposition~\ref{q1q}.
\end{proof}


\end{document}